\documentclass[a4paper,10pt]{amsart}

\usepackage{amsmath}
\usepackage{amsthm}
\usepackage{amssymb}
\usepackage{amsfonts}
\usepackage[bookmarks=true,hyperindex,pdftex,colorlinks,citecolor=blue]{hyperref}

\DeclareMathOperator{\dist}{dist}                           
\DeclareMathOperator{\lspan}{span}                          
\DeclareMathOperator{\supp}{supp}                           
\DeclareMathOperator{\ext}{ext}                             
\DeclareMathOperator{\Lip}{Lip}                             
\DeclareMathOperator{\lip}{lip}                             

\newcommand{\NN}{\mathbb{N}}                                
\newcommand{\RR}{\mathbb{R}}                                

\newcommand{\abs}[1]{\left|{#1}\right|}                     
\newcommand{\pare}[1]{\left({#1}\right)}                    
\newcommand{\set}[1]{\left\{{#1}\right\}}                   
\newcommand{\norm}[1]{\left\|{#1}\right\|}                  
\newcommand{\dual}[1]{{#1}^\ast}                            
\newcommand{\ddual}[1]{{#1}^{\ast\ast}}                     
\newcommand{\ball}[1]{B_{{#1}}}                             
\newcommand{\duality}[1]{\left<{#1}\right>}                 
\newcommand{\cl}[1]{\overline{#1}}                          
\newcommand{\wscl}[1]{\overline{#1}^{\dual{w}}}             
\newcommand{\wsconv}{\stackrel{\dual{w}}{\longrightarrow}}  

\newcommand{\lipfree}[1]{\mathcal{F}({#1})}                 
\newcommand{\lipnorm}[1]{\norm{#1}_L}                       
\newcommand{\mol}[1]{m_{#1}}                                
\newcommand{\meas}[1]{\mathcal{M}({#1})}                    
\newcommand{\rcomp}[1]{{#1}^\mathcal{R}}                    
\newcommand{\bwt}[1]{\beta\widetilde{#1}}                   
\newcommand{\rr}{\mathfrak{r}}                              
\newcommand{\pp}{\mathfrak{p}}                              

\theoremstyle{plain}
\newtheorem{theorem}{Theorem}
\newtheorem{lemma}[theorem]{Lemma}
\newtheorem{corollary}[theorem]{Corollary}
\newtheorem{proposition}[theorem]{Proposition}
\newtheorem*{question*}{Question}

\theoremstyle{definition}
\newtheorem*{definition*}{Definition}

\newtheorem{example}[theorem]{Example}

\begin{document}

\title[Extreme points in Lipschitz-free spaces over compacts]{Extreme points in Lipschitz-free spaces over compact metric spaces}

\author[R. J. Aliaga]{Ram\'on J. Aliaga}
\address[R. J. Aliaga]{Instituto Universitario de Matem\'atica Pura y Aplicada,
Universitat Polit\`ecnica de Val\`encia,
Camino de Vera S/N,
46022 Valencia, Spain}
\email{raalva@upvnet.upv.es}

\date{} 


\begin{abstract}
We prove that all extreme points of the unit ball of a Lipschitz-free space over a compact metric space have finite support. Combined with previous results, this completely characterizes extreme points and implies that all of them are also extreme points in the bidual ball. For the proof, we develop some properties of an integral representation of functionals on Lipschitz spaces originally due to K. de Leeuw.
\end{abstract}




\maketitle


\section{Introduction}

This short paper deals with the extremal structure of the unit ball of Lipschitz-free Banach spaces. Let us begin by establishing our work setting. Given a complete pointed metric space $(M,d)$, where we choose an arbitrary element $0\in M$ as a base point, we consider the Lipschitz space $\Lip_0(M)$ of all real-valued Lipschitz functions $f$ on $M$ such that $f(0)=0$. This space is a Banach space endowed with the norm $\lipnorm{f}$ given by the Lipschitz constant of $f$. More than that, $\Lip_0(M)$ is (isometrically) a dual Banach space. Its canonical predual $\lipfree{M}$ can be constructed by considering the evaluation operators $\delta(x)\colon f\mapsto f(x)$ for $x\in M$, which are obviously elements of $\dual{\Lip_0(M)}$, and then taking
\begin{equation*}
\lipfree{M}=\cl{\lspan}\,\delta(M) .
\end{equation*}
The space $\lipfree{M}$ is usually called the \emph{Lipschitz-free space} over $M$, and the weak$^\ast$ topology induced by it agrees with the topology of pointwise convergence on bounded subsets of $\Lip_0(M)$. Note also that $\delta:M\rightarrow\lipfree{M}$ is an isometric embedding.

Lipschitz and Lipschitz-free spaces are objects of intrinsic interest, as they can be regarded as the metric versions of the usual $C(K)$ spaces and their duals; see Weaver's monograph \cite{Weaver2} for an exhaustive account of their basic properties (note that Lipschitz-free spaces are referred to as \emph{Arens-Eells spaces} in it). However, their current main interest to Banach Space theory is the following linearization property: any Lipschitz mapping $M_1\rightarrow M_2$ between two metric spaces can be extended to a bounded linear operator between $\lipfree{M_1}$ and $\lipfree{M_2}$ (we identify each $M_i$ with $\delta(M_i)\subset\lipfree{M_i}$). See \cite{Godefroy_2015} for a survey of applications of this principle to nonlinear functional analysis.

The roots of the study of the extremal structure of the unit ball $\ball{\lipfree{M}}$ can be traced back to the 1990's. The first important general result was obtained by Weaver \cite{Weaver_1995}: every preserved extreme point of $\ball{\lipfree{M}}$ (i.e. such that it is also an extreme point of the bidual ball $\ball{\ddual{\lipfree{M}}}$) is an \emph{elementary molecule}, i.e. an element of $\lipfree{M}$ of the form
\begin{equation*}
\mol{pq}=\frac{\delta(p)-\delta(q)}{d(p,q)}
\end{equation*}
for some $p\neq q\in M$. As an easy consequence, all finitely supported extreme points must be elementary molecules, too. Note that elementary molecules are the canonical norming elements for $\Lip_0(M)$, which renders them natural suspects for extreme points of the unit ball (compare to the extreme points of $\ball{\dual{C(K)}}$ being multiples of the point masses). Moreover, their simplicity allows for elegant geometric characterizations of extremal elements in terms of an equivalent condition on the pair $p,q$ that only involves the metric structure of $M$.

After Weaver's early result almost no progress was made in this topic for about 20 years. The first new results to appear were due to Garc\'ia-Lirola, Petitjean, Proch\'azka and Rueda Zoca, who proved that the sets of preserved extreme points and denting points of $\ball{\lipfree{M}}$ coincide \cite{GPPR_2018}, and that an elementary molecule $\mol{pq}$ is a strongly exposed point of $\ball{\lipfree{M}}$ if and only if there is $C>0$ such that
\begin{equation*}
d(p,x)+d(q,x)-d(p,q) \geq C\cdot\min\set{d(p,x),d(q,x)}
\end{equation*}
for all $x\in M$ \cite{GPR_2018}. Shortly afterwards the author of this paper, in collaboration with E. Perneck\'a \cite{AP_2020_rmi} and A. J. Guirao \cite{AG_2019}, respectively, gave similar characterizations of elementary molecules that are extreme or preserved extreme points. Namely, $\mol{pq}$ is an extreme, resp. preserved extreme point of $\ball{\lipfree{M}}$ if and only if every point $x\neq p,q$ in $M$, resp. $\beta M$ (the Stone-\v{C}ech compactification of $M$) satisfies
\begin{equation*}
d(p,x)+d(q,x)-d(p,q)>0 .
\end{equation*}
Finally, Petitjean and Proch\'azka proved that any extreme molecule is also an exposed point of $\ball{\lipfree{M}}$; this result appeared in the joint paper \cite{APPP_2020}.

These results complete the characterization of extreme molecules in Lipschitz-free spaces, which covers all preserved extreme points and stronger notions. However, the problem of determining all extreme points of $\ball{\lipfree{M}}$ is not completely solved yet because the following question still remains open:

\begin{question*}
Must every extreme point of $\ball{\lipfree{M}}$ be an elementary molecule?
\end{question*}

Only partial answers to this question have been obtained so far. For instance, this was recently shown to be true for all $0$-hyperbolic metric spaces $M$ in \cite{APP_2021}. It has also been known since the 1990's that the question has a positive answer for compact spaces $M$ such that the space $\lip_0(M)$ of locally flat Lipschitz functions on $M$ is an isometric predual of $\lipfree{M}$ (see \cite[Corollary 4.41]{Weaver2}). Very recently, this condition has been shown to be equivalent to $\lipfree{M}$ simply being a dual space \cite[Theorem B]{AGPP_2021}. This includes compact spaces with null 1-Hausdorff measure or with a snowflaked metric.

Here we provide a positive answer to the question for all compact spaces $M$, with no additional assumptions about duality. In fact, our answer covers the slightly more general case where $M$ is a \emph{proper} space, i.e. such that every closed and bounded subset is compact.

\begin{theorem}
\label{th:extreme_proper}
Let $M$ be a proper metric space. Then the extreme points of $\ball{\lipfree{M}}$ are exactly the elementary molecules $\mol{pq}$ where $p\neq q\in M$ are such that
\begin{equation*}
d(p,x)+d(q,x)>d(p,q)
\end{equation*}
for all $x\in M\setminus\set{p,q}$. Moreover, all of them are also exposed points and preserved extreme points of $\ball{\lipfree{M}}$.
\end{theorem}

\medskip

For the rest of this paper, $M$ will denote a complete metric space with metric $d$ (which will not be proper unless specified explicitly) and we will assume without mention that an arbitrary base point $0\in M$ has been chosen. We will write $B(x,r)$ for the closed ball with center $x$ and radius $r$. The set of extreme points of a convex subset $C$ of a normed space will be denoted as $\ext C$.

By the linearization property mentioned above, the Lipschitz-free space $\lipfree{N}$ over a subset $N\subset M$ such that $0\in N$ can be identified with a closed subspace of $\lipfree{M}$, namely $\lipfree{N}=\cl{\lspan}\,\delta(N)$. Let us recall that, for $m\in\lipfree{M}$, its \emph{support} $\supp(m)$ is defined as the smallest closed subset of $M$ such that $m\in\lipfree{\supp(m)\cup\set{0}}$. The existence of such a set is proved in \cite{AP_2020_rmi,APPP_2020}.

For a compact Hausdorff space $X$, $\meas{X}$ will stand for the space of Radon measures (i.e. finite regular signed Borel measures) on $X$ with the total variation norm. Recall that, given two compact Hausdorff spaces $X,Y$, a continuous mapping $f:X\rightarrow Y$, and a measure $\mu\in\meas{X}$, the \emph{pushforward measure} $f_\sharp\mu$ is defined by
\begin{equation*}
f_\sharp\mu(E)=\mu(f^{-1}(E))
\end{equation*}
for every Borel $E\subset Y$. We have $f_\sharp\mu\in\meas{Y}$ and $\int_Y g\,d(f_\sharp\mu)=\int_X (g\circ f)\,d\mu$ for every Borel function $g$ on $Y$. Moreover, if $\mu$ is positive then $f_\sharp\mu$ is clearly also positive and $\norm{f_\sharp\mu}=\norm{\mu}$. See e.g. \cite[Sections 3.6 and 9.1]{Bogachev} for reference.

\section{Integral representation}

Our proof of Theorem \ref{th:extreme_proper} will be based on the integral representation of functionals in $\lipfree{M}$ and its bidual. In \cite{AP_2020_measures}, the author and E. Perneck\'a studied the functionals that can be expressed as integration against a measure on $M$ or some compactification of $M$. In general this is not possible for all elements of $\lipfree{M}$, only for those that can be written as the difference of two positive elements. But there is an alternative representation, originally due to de Leeuw \cite{deLeeuw_1961}, that covers all functionals in $\dual{\Lip_0(M)}$. The key idea is representing a functional not as an ``integral of evaluation functionals'', but as an ``integral of elementary molecules'' instead.

Let us introduce this representation. Following \cite{Weaver2}, denote
\begin{equation*}
\widetilde{M}=\set{(x,y)\in M\times M:x\neq y}
\end{equation*}
with the topology inherited from $M\times M$. The \emph{de Leeuw map} is the mapping $\Phi:\RR^M\rightarrow\RR^{\widetilde{M}}$ defined by
\begin{equation*}
\Phi f(x,y)=\frac{f(x)-f(y)}{d(x,y)}
\end{equation*}
for $f\in\RR^M$ and $(x,y)\in\widetilde{M}$. Clearly $\norm{\Phi f}_\infty=\lipnorm{f}$. Therefore, for any Lipschitz $f$, the function $\Phi f$ is continuous and bounded and thus admits a unique continuous extension to $\bwt{M}$, the Stone-\v{C}ech compactification of $\widetilde{M}$, that we will still denote by $\Phi f$ for simplicity. Thus we may consider $\Phi$ as a linear isometric embedding of $\Lip_0(M)$ into $C(\bwt{M})$. Its adjoint operator $\dual{\Phi}$ is therefore a quotient map from $C(\bwt{M})^{\ast}=\meas{\bwt{M}}$ onto $\dual{\Lip_0(M)}=\ddual{\lipfree{M}}$. That is, for every $\phi\in\ddual{\lipfree{M}}$ there exists a representing measure $\mu\in\meas{\bwt{M}}$ such that $\dual{\Phi}\mu=\phi$, i.e.
\begin{equation*}
\duality{f,\phi} = \int_{\bwt{M}} (\Phi f)\,d\mu
\end{equation*}
for any $f\in\Lip_0(M)$. It should be obvious that the elementary molecule $\mol{xy}$ is represented by the Dirac measure $\delta_{(x,y)}$, for any $(x,y)\in\widetilde{M}$.

The following mappings defined on $\bwt{M}$ will be used often:
\begin{itemize}
\item Let $\pp:\widetilde{M}\rightarrow M\times M$ be the identity mapping and extend it to a continuous mapping $\pp:\bwt{M}\rightarrow\beta M\times\beta M$. We will call this the \emph{projection} mapping, and denote its first and second coordinates by $\pp_1$ and $\pp_2$, respectively. For a set $E\subset\bwt{M}$ we will denote $\pp_s(E)=\pp_1(E)\cup\pp_2(E)\subset\beta M$ and call this the \emph{shadow} of $E$.
\item Let $\rr:\widetilde{M}\rightarrow\widetilde{M}$ be the \emph{reflection} mapping given by $\rr(x,y)=(y,x)$, and extend it to a continuous mapping $\rr:\bwt{M}\rightarrow\bwt{M}$. Notice that $\rr\circ\rr$ is the identity on $\widetilde{M}$ and hence on all of $\bwt{M}$, thus $\rr$ is a bijection.
\end{itemize}

The representing measure $\mu\in\meas{\bwt{M}}$ for $\phi\in\ddual{\lipfree{M}}$ is not unique. Since $\dual{\Phi}$ is a quotient map, we may always choose it such that $\norm{\mu}=\norm{\phi}$. We will now see that we may additionally choose $\mu$ to be positive. The underlying idea is that we may replace a negative point mass $-\delta_{(x,y)}$ by the positive point mass $\delta_{(y,x)}$ as both represent the same functional $\mol{yx}$. We may formalize this in terms of pushforward measures as follows.

\begin{lemma}
\label{lm:pushforward}
For any $\mu\in\meas{\bwt{M}}$ we have $\dual{\Phi}(\mu+\rr_\sharp\mu)=0$.
\end{lemma}

\begin{proof}
If $\mu=\delta_{(x,y)}$ for $(x,y)\in\widetilde{M}$, the identity holds as $\dual{\Phi}\delta_{(x,y)}=\mol{xy}$ and
\begin{equation*}
\dual{\Phi}(\rr_\sharp\delta_{(x,y)})=\dual{\Phi}\delta_{(y,x)}=\mol{yx}=-\mol{xy} .
\end{equation*}
Now notice that
\begin{equation*}
\meas{\bwt{M}} = \wscl{\lspan}\ext\ball{\meas{\bwt{M}}} = \wscl{\lspan}\set{\delta_\zeta:\zeta\in\bwt{M}} = \wscl{\lspan}\set{\delta_\zeta:\zeta\in\widetilde{M}} .
\end{equation*}
Indeed, the first equality follows from the Krein-Milman theorem, the second one from the well-known fact that the extreme points of $\ball{\meas{\bwt{M}}}$ are the positive and negative Dirac measures, and the last one from the defining properties of the Stone-\v{C}ech compactification $\bwt{M}$. The identity now follows for all $\mu\in\meas{\bwt{M}}$ by the weak$^\ast$ continuity of $\dual{\Phi}$.
\end{proof}

\begin{proposition}
\label{eq:norm_attaining_representation}
For every $\phi\in\ddual{\lipfree{M}}$ there is a positive $\mu\in\meas{\bwt{M}}$ such that $\dual{\Phi}\mu=\phi$ and $\norm{\mu}=\norm{\phi}$.
\end{proposition}

\begin{proof}
Since $\dual{\Phi}$ is a quotient map, we may find $\mu$ that satisfies all requirements except possibly being positive. In order to achieve that, let $\mu=\mu^+-\mu^-$ be the Jordan decomposition of $\mu$ and let $\mu'=\mu^++\rr_\sharp\mu^-$. By Lemma \ref{lm:pushforward} we have $0=\dual{\Phi}(\mu^-+\rr_\sharp\mu^-)$, therefore
\begin{equation*}
\dual{\Phi}(\mu') = \dual{\Phi}\mu^++\dual{\Phi}(\rr_\sharp\mu^-) = \dual{\Phi}\mu^+-\dual{\Phi}\mu^- = \dual{\Phi}\mu
\end{equation*}
as we needed. It is clear that $\mu'\geq 0$ and
\begin{equation*}
\norm{\mu'}=\norm{\mu^+}+\norm{\rr_\sharp\mu^-}=\norm{\mu^+}+\norm{\mu^-}=\norm{\mu} . \qedhere
\end{equation*}
\end{proof}

In general, measures $\mu\in\meas{\bwt{M}}$ represent functionals $\dual{\Phi}\mu$ that belong to $\ddual{\lipfree{M}}$. In our argument, we will need to determine whether a given measure $\mu$ represents a weak$^\ast$ continuous functional on $\Lip_0(M)$. This is usually not an easy task. A natural sufficient condition is given by the following result:

\begin{proposition}[{\cite[Lemma 4.36]{Weaver2}}]
\label{pr:measure_in_fm}
If $\mu\in\meas{\bwt{M}}$ is concentrated on $\widetilde{M}$ then $\dual{\Phi}\mu\in\lipfree{M}$.
\end{proposition}

\noindent Note that this result is stated in \cite{Weaver2} under an additional hypothesis of local compactness, but the proof does not use this extra hypothesis at all and is valid for the general case. For completeness, we provide below a slightly simplified proof of this result. Note also that $\widetilde{M}$ is a Borel subset of $\bwt{M}$ when $M$ is complete, in fact it is a $G_\delta$ set. Indeed, $\widetilde{M}=\pp^{-1}(M\times M)\setminus d^{-1}(0)$ where $d:\bwt{M}\rightarrow [0,\infty]$ is the continuous extension of the metric, and $M$ is well-known to be a $G_\delta$ subset of $\beta M$ when complete.

\begin{proof}[Proof of Proposition \ref{pr:measure_in_fm}]
We will show that $\dual{\Phi}\mu$ is weak$^\ast$ continuous. By the Banach-Dieudonn\'e theorem, it is enough to check that its restriction to $\ball{\Lip_0(M)}$ is weak$^\ast$ continuous. That is: fix a net $(f_i)$ in $\ball{\Lip_0(M)}$ that converges weak$^\ast$, i.e. pointwise, to $f\in\ball{\Lip_0(M)}$, then we need to show that $\duality{f_i,\dual{\Phi}\mu}\rightarrow\duality{f,\dual{\Phi}\mu}$.

Suppose first that $M$ is separable. Then $\ball{\Lip_0(M)}$ is weak$^\ast$ metrizable, so it is enough to prove the result for a sequence $(f_n)$ instead of a net. In that case we have $\Phi f_n\rightarrow\Phi f$ pointwise on $\widetilde{M}$, and $\abs{\Phi f_n}\leq 1$ which is $\mu$-integrable, hence
\begin{equation*}
\lim_{n\rightarrow\infty}\duality{f_n,\dual{\Phi}\mu} = \lim_{n\rightarrow\infty}\int_{\widetilde{M}}(\Phi f_n)\,d\mu = \int_{\widetilde{M}}(\Phi f)\,d\mu = \duality{f,\dual{\Phi}\mu}
\end{equation*}
by Lebesgue's dominated convergence theorem.

For the general case, fix $\varepsilon>0$. By regularity there is a compact set $K\subset\widetilde{M}$ such that $\norm{\mu-\mu|_K}=\abs{\mu-\mu|_K}(\widetilde{M})\leq\varepsilon$, and so $\norm{\dual{\Phi}\mu-\dual{\Phi}(\mu|_K)}\leq\varepsilon$. Then $S=\pp_s(K)$ is a compact subset of $M$, hence separable. Since $K\subset\widetilde{S}$ we may identify $\mu|_K$ with a Radon measure on $\widetilde{S}$. Obviously $f_i|_S\rightarrow f|_S$ pointwise, therefore
\begin{equation*}
\lim_i\duality{f_i,\dual{\Phi}(\mu|_K)} = \lim_i\int_{\widetilde{S}}(\Phi f_i)\,d\mu|_K = \int_{\widetilde{S}}(\Phi f)\,d\mu|_K = \duality{f,\dual{\Phi}(\mu|_K)}
\end{equation*}
where the second equality holds by the previous paragraph. Thus $\dual{\Phi}(\mu|_K)\in\lipfree{M}$ and so $\dist(\dual{\Phi}\mu,\lipfree{M})\leq\varepsilon$. Letting $\varepsilon\rightarrow 0$ yields $\dual{\Phi}\mu\in\lipfree{M}$.
\end{proof}

The converse of Proposition \ref{pr:measure_in_fm} is false. For instance, take any $\zeta\in\bwt{M}\setminus\widetilde{M}$, then $\mu$ and $\mu+\delta_\zeta+\delta_{\rr(\zeta)}$ represent the same functional by Lemma \ref{lm:pushforward}. But even if we avoid this trick by restricting our attention to measures with minimal norm, i.e. satisfying the conditions in Proposition \ref{eq:norm_attaining_representation}, we can still find counterexamples to Proposition \ref{pr:measure_in_fm}:

\begin{example}
Consider $M=[0,1]$ and $m=\delta(1)\in\lipfree{M}$. For every $n\in\NN$, let
\begin{equation*}
\mu_n=\frac{1}{n}\sum_{k=1}^n \delta_{\pare{\frac{k}{n},\frac{k-1}{n}}}\in\meas{\bwt{M}} .
\end{equation*}
It should be clear that $\norm{\mu_n}=1$ and $\dual{\Phi}\mu_n=m$ for every $n$. Now let $\mu$ be a weak$^\ast$ cluster point of the sequence $(\mu_n)$. Then $\dual{\Phi}\mu=m$, but $\supp(\mu)$ is clearly contained in the ``diagonal'' $\bwt{M}\setminus\widetilde{M}$.
\end{example}

While it is in general not possible to guarantee that a measure $\mu$ representing an element of $\lipfree{M}$ will be concentrated on $\widetilde{M}$, we can at least prove that it will be concentrated ``away from infinity''. More specifically, for proper $M$ we can prove that $\mu$ is concentrated on $\pp^{-1}(M\times M)$. In the proof we will require the following simple fact.

\begin{lemma}
\label{lm:phi_bounded_support}
Suppose that $M$ is proper and let $h$ be a Lipschitz function on $M$ with bounded support. Then $\Phi h$ vanishes outside of $\pp^{-1}(M\times M)$.
\end{lemma}

\begin{proof}
Let $\zeta\in\bwt{M}\setminus\pp^{-1}(M\times M)$, choose a net $(x_i,y_i)$ in $\widetilde{M}$ that converges to $\zeta$, and select a subnet such that $d(x_i,y_i)$ either tends to infinity or converges to a finite value. In the former case $\abs{\Phi h(x_i,y_i)}\leq 2\norm{h}_\infty/d(x_i,y_i)$ converges to $0$ and thus $\Phi h(\zeta)=0$. In the latter case, notice that either both $\pp_1(\zeta)$ and $\pp_2(\zeta)$ belong to $M$, or neither of them does. The first case implies $\zeta\in\pp^{-1}(M\times M)$, whereas the second case implies $h(x_i)=h(y_i)=0$ eventually as $h$ vanishes outside of a bounded set, and so $\Phi h(\zeta)=0$.
\end{proof}

\begin{proposition}
\label{pr:finite_concentration}
Suppose that $M$ is proper and $\mu\in\meas{\bwt{M}}$ is such that $\dual{\Phi}\mu\in\lipfree{M}$, $\norm{\dual{\Phi}\mu}=\norm{\mu}$, and $\mu\geq 0$. Then $\mu$ is concentrated on $\pp^{-1}(M\times M)$.
\end{proposition}

\begin{proof}
For $n\in\NN$, let $H_n$ be the function defined by
\begin{equation*}
H_n(x) = \begin{cases}
1 &\text{, if } d(x,0)\leq 2^n \\
2-2^{-n}d(x,0) &\text{, if } 2^n\leq d(x,0)\leq 2^{n+1} \\
0 &\text{, if } 2^{n+1}\leq d(x,0)
\end{cases}
\end{equation*}
for $x\in M$. By \cite[Lemma 2.3]{APPP_2020}, for every $f\in\Lip_0(M)$ we have $fH_n\in\Lip_0(M)$ with $\lipnorm{fH_n}\leq 3\lipnorm{f}$ and therefore $fH_n\wsconv f$. It follows that
\begin{equation}
\label{eq:limit_Hn}
\int_{\bwt{M}} (\Phi f)\,d\mu = \lim_{n\rightarrow\infty} \int_{\bwt{M}} \Phi(fH_n)\,d\mu = \lim_{n\rightarrow\infty} \int_S \Phi(fH_n)\,d\mu
\end{equation}
where $S=\pp^{-1}(M\times M)$. Indeed, the first equality holds because $\dual{\Phi}\mu$ is weak$^\ast$ continuous and the second one is a consequence of Lemma \ref{lm:phi_bounded_support} as each $fH_n$ has bounded support.

We now claim that $\Phi(fH_n)\rightarrow\Phi f$ pointwise on $S$. To see this, fix $\zeta\in S$ and let $N\in\NN$ be such that $\pp_1(\zeta),\pp_2(\zeta)\in B(0,2^N)$. Then, for $n>N$, we have $H_n=1$ on neighborhoods of $\pp_1(\zeta)$ and $\pp_2(\zeta)$ . Thus, if $(x_i,y_i)$ is a net in $\widetilde{M}$ that converges to $\zeta$, we have $x_i\rightarrow\pp_1(\zeta)$, $y_i\rightarrow\pp_2(\zeta)$ and so
\begin{align*}
\Phi(fH_n)(\zeta) = \lim_i\Phi(fH_n)(x_i,y_i) &= \lim_i\frac{f(x_i)H_n(x_i)-f(y_i)H_n(y_i)}{d(x_i,y_i)} \\
&= \lim_i\frac{f(x_i)-f(y_i)}{d(x_i,y_i)} = \Phi f(\zeta)
\end{align*}
therefore $\Phi(fH_n)(\zeta)=\Phi f(\zeta)$. Lebesgue's dominated convergence theorem then implies that the limit \eqref{eq:limit_Hn} equals $\int_S (\Phi f)\,d\mu$. We conclude that $\mu$ and $\mu|_S$ represent the same functional in $\lipfree{M}$. But then $\norm{\mu}=\norm{\dual{\Phi}\mu}=\norm{\dual{\Phi}(\mu|_S)}\leq\norm{\mu|_S}$, and so $\mu$ must be concentrated on $S$.
\end{proof}

We will also need some understanding of the relationship between the support of $\mu$ and the support of the represented functional when the latter belongs to $\lipfree{M}$.

\begin{lemma}
\label{lm:support_inclusion}
Let $m\in\lipfree{M}$ and $\mu\in\meas{\bwt{M}}$ be such that $\dual{\Phi}\mu=m$. Then $\supp(m)\subset\pp_s(\supp(\mu))$.
\end{lemma}

\begin{proof}
Recall that by \cite[Proposition 2.7]{APPP_2020} a point $x\in M$ belongs to $\supp(m)$ if and only if for every $r>0$ there exists a function $f\in\Lip_0(M)$ that is supported on $B(x,r)$ and such that $\duality{m,f}\neq 0$. We will show that no point $x\in M\setminus\pp_s(\supp(\mu))$ satisfies that condition, and this will prove the lemma.

Indeed, fix one such $x$. Notice that, since $\supp(\mu)$ is a compact subset of $\bwt{M}$, $\pp_s(\supp(\mu))$ is also a compact subset of $\beta M$. Therefore there is $r>0$ such that $K=\cl{B(x,r)}^{\beta M}$ does not intersect $\pp_s(\supp(\mu))$. Suppose that $f\in\Lip_0(M)$ is supported in $B(x,r)$ and let
\begin{equation*}
W = \pp_1^{-1}(K) \cup \pp_2^{-1}(K)
\end{equation*}
which is a compact subset of $\bwt{M}$ that does not intersect $\supp(\mu)$. If $\zeta\in\supp(\mu)$ then there is a neighborhood $U$ of $\zeta$ such that $U\cap W=\varnothing$. Hence, if $(x,y)\in\widetilde{M}\cap U$ then $x,y\notin B(x,r)$ and $f(x)=f(y)=0$, and it follows by density that $\Phi f(\zeta)=0$. This shows that $\Phi f=0$ on $\supp(\mu)$ and therefore $\duality{m,f}=\int_{\bwt{M}}(\Phi f)\,d\mu=0$. As explained above, this proves that $x\notin\supp(m)$.
\end{proof}

Since the shadow of $\supp(\mu)$ is a compact subset of $\beta M$, by Lemma \ref{lm:support_inclusion} it must always contain the closure $\cl{\supp(m)}^{\beta M}$. It is easy to see that, in fact, we can always find a representing measure whose support has minimal shadow. Note that it might be necessary to include the base point in this minimal shadow, as $0$ will not belong to $\supp(m)$ if it is isolated in $\pp_s(\supp(\mu))$.

\begin{proposition}
\label{pr:minimal_support_representation}
For every $m\in\lipfree{M}$ there is a positive $\mu\in\meas{\bwt{M}}$ such that $\dual{\Phi}\mu=m$, $\norm{\mu}=\norm{m}$, and
\begin{equation*}
\cl{\supp(m)}^{\beta M} \subset \pp_s(\supp(\mu)) \subset \cl{\supp(m)}^{\beta M}\cup\set{0} .
\end{equation*}
\end{proposition}

\begin{proof}
Let $S=\supp(m)\cup\set{0}$, then $m\in\lipfree{S}$ and so Proposition \ref{eq:norm_attaining_representation} yields a positive measure $\mu\in\meas{\bwt{S}}$ with $\norm{\mu}=\norm{m}$ and such that $\dual{(\Phi|_{\Lip_0(S)})}\mu=m$. It is easily checked that $\bwt{S}$ and the closure $\cl{\widetilde{S}}^{\bwt{M}}$ of $\widetilde{S}$ in $\bwt{M}$ are equivalent compactifications of $\widetilde{S}$, and that $\beta S$ can be similarly identified with $\cl{S}^{\beta M}$. Therefore $\mu$ can be identified with an element of $\meas{\bwt{M}}$ that is supported on $\cl{\widetilde{S}}^{\bwt{M}}$ and such that $\dual{\Phi}\mu=m$. Clearly $\pp_s(\supp(\mu))\subset\beta S=\cl{S}^{\beta M}$ so the rightmost inclusion holds for this $\mu$, whereas the leftmost inclusion is implied by Lemma \ref{lm:support_inclusion}.
\end{proof}

The representation from Proposition \ref{pr:minimal_support_representation} is not unique in general, even when the base point may be disregarded. For instance, let $M=\set{0,1,2,3}\subset\RR$ and $m=\delta(2)+\delta(3)-2\delta(1)$. It is straightforward to check that $\norm{m}=3$ and $\supp(m)=\set{1,2,3}$. Then $\mu=\delta_{(2,1)}+2\delta_{(3,1)}$ and $\mu'=2\delta_{(2,1)}+\delta_{(3,2)}$ are two different positive representations of $m$ with minimal norm and minimal shadow.

Let us remark here that many of the previous statements admit generalizations. For instance, when $M$ is not proper, the same arguments show that Lemma \ref{lm:phi_bounded_support} and Proposition \ref{pr:finite_concentration} hold with $\pp^{-1}(\rcomp{M}\times\rcomp{M})$ in place of $\pp^{-1}(M\times M)$, where $\rcomp{M}$ is the subset of $\beta M$ consisting of points that are limits of bounded nets. For proper $M$, that set is simply $M$. There are also versions of Lemma \ref{lm:support_inclusion} and Proposition \ref{pr:minimal_support_representation} that hold for functionals in $\ddual{\lipfree{M}}$ that are not weak$^\ast$ continuous, as long as we consider an appropriate definition of support for such functionals (see \cite[Section 3]{AP_2020_measures}). However, introducing this definition is a harder task that requires several new notions and replacing references to $\beta M$ with the lesser-known Samuel (or uniform) compactification of $M$. We have thus decided not to include the details here as they will not be needed for the proof of Theorem \ref{th:extreme_proper} in the proper case.

\section{Proof of the main result}

Let us now proceed with the proof of Theorem \ref{th:extreme_proper}. We start with the following fact:

\begin{lemma}
\label{lm:extreme_measure}
Let $m\in\ext\ball{\lipfree{M}}$ and $\mu\in\meas{\bwt{M}}$ be such that $\dual{\Phi}\mu=m$, $\norm{\mu}=1$, and $\mu\geq 0$. Suppose that $\lambda\in\meas{\bwt{M}}$ is such that $0\leq\lambda\leq\mu$ and $\dual{\Phi}\lambda\in\lipfree{M}$. Then $\dual{\Phi}\lambda=\norm{\lambda}\cdot m$.
\end{lemma}

\begin{proof}
Let $\nu=\mu-\lambda$, so that $\nu\geq 0$ and $\dual{\Phi}\nu\in\lipfree{M}$. If either $\lambda=0$ or $\nu=0$ then the lemma holds trivially, so assume otherwise. By positivity we have $\norm{\lambda}+\norm{\nu}=\norm{\mu}=1$. Using the fact that $\norm{\dual{\Phi}}=1$ we get
\begin{equation*}
1 = \norm{\dual{\Phi}\mu} \leq \norm{\dual{\Phi}\lambda}+\norm{\dual{\Phi}\nu} \leq \norm{\lambda}+\norm{\nu} = 1 .
\end{equation*}
All inequalities are therefore equalities. It follows that $\norm{\dual{\Phi}\lambda'}=\norm{\dual{\Phi}\nu'}=1$, where $\lambda'=\lambda/\norm{\lambda}$ and $\nu'=\nu/\norm{\nu}$. Now notice that
\begin{equation*}
m = \dual{\Phi}\lambda + \dual{\Phi}\nu = \norm{\lambda}\cdot\dual{\Phi}\lambda' + \norm{\nu}\cdot\dual{\Phi}\nu'
\end{equation*}
is a convex combination of $\dual{\Phi}\lambda',\dual{\Phi}\nu'\in\ball{\lipfree{M}}$. Hence $\dual{\Phi}\lambda'=m$ as claimed.
\end{proof}

This lemma tells us the following: if we use Proposition \ref{eq:norm_attaining_representation} to choose a representation $\mu$ of $m\in\ext\ball{\lipfree{M}}$ that is positive and has minimal norm, then every measure $\lambda$ defined by $d\lambda = h\,d\mu$ where $0\leq h\leq 1$ will automatically represent a multiple of the same functional $m$, \emph{as long as we know that it represents an element of $\lipfree{M}$}. It is easy to force such a measure $\lambda$ to be concentrated on certain regions by choosing $h$ appropriately, and this will yield upper bounds for $\supp(m)$ by Lemma \ref{lm:support_inclusion}. The difficulty lies in ensuring that $\lambda$ represents a weak$^\ast$ continuous functional -- the condition from Proposition \ref{pr:measure_in_fm} is too simple to be useful here. The next lemma shows how we can accomplish this when $M$ is proper.

\begin{lemma}
\label{lm:weighted_measure_proper}
Suppose that $M$ is proper, and let $h$ be a Lipschitz function on $M$ with bounded support. Then, for every $\mu\in\meas{\bwt{M}}$ such that $\dual{\Phi}\mu\in\lipfree{M}$, we also have $\dual{\Phi}\lambda\in\lipfree{M}$ where $d\lambda=(h\circ\pp_1)\,d\mu$.
\end{lemma}

By symmetry (or by analogous construction), the same is true if we replace $\pp_1$ with $\pp_2$.

\begin{proof}
Let us start by noticing that the identity
\begin{equation*}
\Phi (fh)(x,y) = \Phi f(x,y)\cdot h(x) + \Phi h(x,y)\cdot f(y)
\end{equation*}
holds for any $f\in\Lip_0(M)$ and any $(x,y)\in\widetilde{M}$. It follows that the continuous extensions
\begin{equation*}
\Phi (fh) = (\Phi f)\cdot (h\circ\pp_1) + (\Phi h)\cdot (f\circ\pp_2)
\end{equation*}
also agree in $\bwt{M}$. Since $fh\in\Lip_0(M)$, we may integrate against $\mu$ to obtain
\begin{align*}
\duality{fh,\dual{\Phi}\mu} &= \int_{\bwt{M}}(\Phi f)(h\circ\pp_1)\,d\mu + \int_{\bwt{M}}(f\circ\pp_2)(\Phi h)\,d\mu \\
&= \duality{f,\dual{\Phi}\lambda} + \int_{\bwt{M}}(f\circ\pp_2)\,d\nu \\
&= \duality{f,\dual{\Phi}\lambda} + \int_{\beta M} f\,d((\pp_2)_\sharp\nu)
\end{align*}
where $d\nu=(\Phi h)\,d\mu$. Notice that $\lambda,\nu\in\meas{\bwt{M}}$, as $h\circ\pp_1$ and $\Phi h$ are continuous and bounded. By Lemma \ref{lm:phi_bounded_support}, $\nu$ is in fact concentrated on $\pp^{-1}(M\times M)$ and therefore $(\pp_2)_\sharp\nu$ is concentrated on $M$.

Recall that the functional $f\mapsto\duality{fh,\dual{\Phi}\mu}$ is weak$^\ast$ continuous (see e.g. \cite[Lemma 2.3]{APPP_2020}). Therefore, in order to prove that $\dual{\Phi}\lambda\in\lipfree{M}$ it will be enough to show that the functional
\begin{equation*}
f \mapsto \int_{\beta M} f\,d((\pp_2)_\sharp\nu) = \int_M f\,d((\pp_2)_\sharp\nu)
\end{equation*}
also belongs to $\lipfree{M}$. By \cite[Proposition 4.4]{AP_2020_measures}, it suffices to show that the function $\rho\in\Lip_0(M)$ given by $\rho(x)=d(x,0)$ is integrable against $(\pp_2)_\sharp\nu$, i.e. that the integral
\begin{equation*}
\int_{\bwt{M}}(\rho\circ\pp_2)(\Phi h)\,d\mu
\end{equation*}
is finite. Thus it will be enough to prove that $g=(\rho\circ\pp_2)\abs{\Phi h}$ is bounded. In the case where $M$ is compact, this is immediate as $\rho$ is bounded. Let us now verify it in the unbounded case. Fix $r>0$ such that $\supp(h)\subset B(0,r)$. Let $(x,y)\in\widetilde{M}$, and consider three cases:
\begin{itemize}
\item If $y\in B(0,2r)$, then $g(x,y)=d(y,0)\abs{\Phi h(x,y)}\leq 2r\lipnorm{h}$.
\item If $y\notin B(0,2r)$ and $x\in B(0,r)$, then $h(y)=0$ and
\begin{equation*}
g(x,y) = d(y,0)\frac{\abs{h(x)}}{d(x,y)} \leq \frac{d(y,0)}{d(x,y)}\cdot r\lipnorm{h} \leq \pare{1+\frac{d(x,0)}{d(x,y)}}\cdot r\lipnorm{h} \leq 2r\lipnorm{h} .
\end{equation*}
\item If $y\notin B(0,2r)$ and $x\notin B(0,r)$, then $h(y)=h(x)=0$ and $g(x,y)=0$.
\end{itemize}
Hence $\norm{g}_\infty\leq 2r\lipnorm{h}$ on $\widetilde{M}$ and thus also on $\bwt{M}$. This finishes the proof.
\end{proof}

We can now finally prove Theorem \ref{th:extreme_proper}, following the argument outlined before Lemma \ref{lm:weighted_measure_proper}.

\begin{proof}[Proof of Theorem \ref{th:extreme_proper}]
Let $m\in\ext\ball{\lipfree{M}}$. Apply Proposition \ref{eq:norm_attaining_representation} to find $\mu\in\meas{\bwt{M}}$ such that $\dual{\Phi}\mu=m$, $\norm{\mu}=1$ and $\mu\geq 0$. By Proposition \ref{pr:finite_concentration}, $\mu$ is concentrated on the open set $\pp^{-1}(M\times M)$, therefore $\supp(\mu)$ must intersect that set. Hence we may choose a point $\zeta\in\supp(\mu)$ such that $x=\pp_1(\zeta)$ and $y=\pp_2(\zeta)$ belong to $M$ (we allow the case $x=y$). Fix $z\in M\setminus\set{x,y}$, and we will show that $z\notin\supp(m)$.

Construct a Lipschitz function $h$ on $M$ with bounded support and such that $0\leq h\leq 1$, $h=1$ in a neighborhood of $x$, and $h=0$ in a neighborhood of $z$. If $\lambda\in\meas{\bwt{M}}$ is given by $d\lambda=(h\circ\pp_1)\,d\mu$, then $0\leq\lambda\leq\mu$ and moreover $\dual{\Phi}\lambda\in\lipfree{M}$ by Lemma \ref{lm:weighted_measure_proper}. It is clear that $\lambda\neq 0$ as $\zeta\in\supp(\lambda)$. Thus, writing $\lambda'=\lambda/\norm{\lambda}$, we have $\dual{\Phi}\lambda'=m$ by Lemma \ref{lm:extreme_measure}. We may therefore replace our representing measure $\mu$ with $\lambda'$, and hence assume that $\zeta\in\supp(\mu)$ but $z\notin\pp_1(\supp(\mu))$.

A similar construction with $y$ and $\pp_2$ in place of $x$ and $\pp_1$, respectively, shows that we may replace $\mu$ again and assume that $\zeta\in\supp(\mu)$ but $z\notin\pp_2(\supp(\mu))$. This implies that $z\notin\pp_s(\supp(\mu))$ and therefore $z\notin\supp(m)$ by Lemma \ref{lm:support_inclusion}. Since $z$ was arbitrary, we have thus proved that $\supp(m)\subset\set{x,y}$ is finite.

The main result now follows from \cite[Theorem 1.1]{AP_2020_rmi}, which characterizes extreme points with finite support. The fact that $m$ is an exposed point follows from \cite[Theorem 3.2]{APPP_2020}, and it is a preserved extreme point by either \cite[Theorem 4.2]{AG_2019} for the compact case or \cite[Proposition 2.3(a)]{AP_2020_rmi} for the proper case.
\end{proof}

The following corollary is immediate:

\begin{corollary}
Let $M$ be a complete metric space and $m\in\ext\ball{\lipfree{M}}$. If $\supp(m)$ is proper, then $m$ is an elementary molecule.
\end{corollary}

\begin{proof}
Let $S=\supp(m)\cup\set{0}$, then $m\in\lipfree{S}$ by the definition of support and obviously $m\in\ext\ball{\lipfree{S}}$, thus $m$ is a molecule by Theorem \ref{th:extreme_proper}.
\end{proof}

We finish this paper by remarking that most of the arguments developed here are general and work for any choice of complete metric space $M$. The only place where compactness is used in an essential way is Lemma \ref{lm:weighted_measure_proper}. So finding a different argument or construction that yields measures $\lambda\leq\mu$ representing weak$^\ast$ continuous functionals might provide a way to extend Theorem \ref{th:extreme_proper} to the noncompact case.

\section*{Acknowledgments}
We wish to thank Eva Perneck\'a and the anonymous referee for many suggestions and corrections to the original manuscript.

The author was partially supported by the Spanish Ministry of Economy, Industry and Competitiveness under Grant MTM2017-83262-C2-2-P.


\end{document}